\documentclass[12pt, reqno]{amsart}
\usepackage[bookmarksnumbered,colorlinks, plainpages]{hyperref}
\hypersetup{colorlinks=true,linkcolor=red, anchorcolor=green, citecolor=cyan, urlcolor=red, filecolor=magenta, pdftoolbar=true}

\usepackage{subfig}
\usepackage{rotating}
\usepackage{tikz}

\usepackage{fixltx2e}
\usepackage{array}
\RequirePackage{tkz-berge}

\usepackage{amsmath, amsthm, amscd, amsfonts, amssymb, graphicx, color, enumerate, xcolor,  mathrsfs, latexsym}

 \definecolor{cupgreen}{rgb}{0,0.498,0.208}
  \definecolor{cupblue}{rgb}{0,0,.5}
  \definecolor{cupred}{rgb}{1,0.04,0}
  \definecolor{cuppink}{rgb}{0.925,0,0.545}
  \definecolor{cupmagenta}{rgb}{0.624,0.161,0.424}
  \definecolor{cupbrown}{rgb}{0.71,0.212,0.133}

  \definecolor{cupgreen}{rgb}{0,0,0}
  \definecolor{cupblue}{rgb}{0,0,0}
  \definecolor{cupred}{rgb}{0,0,0}
  \definecolor{cuppink}{rgb}{0,0,0}
  \definecolor{cupmagenta}{rgb}{0,0,0}
  \definecolor{cupbrown}{rgb}{0,0,0}

\definecolor{TITLE}{rgb}{0,0,0}
\definecolor{AUTHOR1}{rgb}{0.00,0.59,0.00}
\definecolor{AUTHOR2}{rgb}{0.50,0.00,1.00}
\definecolor{SECTION}{rgb}{0.50,0.00,1.00}
\definecolor{FOOTTITLE}{rgb}{0.00,0.50,0.75}
\definecolor{THM}{rgb}{0.8,0,0.1}
\definecolor{SEC}{rgb}{0,0,1}

\makeatletter

\newtheorem{theorem}{{\color{THM} Theorem}}[section]
\newtheorem{lemma}[theorem]{{\color{THM}Lemma}}
\newtheorem{proposition}[theorem]{{\color{THM}Proposition}}
\newtheorem{corollary}[theorem]{{\color{THM}Corollary}}
\theoremstyle{definition}

\numberwithin{equation}{section}

\usepackage{amscd}
\begin{document}

\newbox\Adr
\setbox\Adr\vbox{
\centerline{\sc Daniel Yaqubi$^1$ and Madjid Mirzavaziri$^2$  }
\vskip18pt
\centerline{$^{1}$Department of Pure Mathematics, Ferdowsi University of Mashhad}
\centerline{P. O. Box 1159, Mashhad 91775, Iran.}
\centerline{Email: {\tt daniel\_yaqubi@yahoo.es } } 
\vskip18pt
\centerline{$^{2}$Department of Pure Mathematics, Ferdowsi University of Mashhad}
\centerline{P. O. Box 1159, Mashhad 91775, Iran.}
\centerline{Email: {\tt mirzavaziri@um.ac.ir}.}} 

\title[Some results on ordered an unordered factrizations]{SOME RESULTS ON ORDERED AND UNORDERED
FACTORIZATIONS OF A POSITIVE INTEGER }

\author[ Daniel Yaqubi and Madjid Mirzavaziri]{\box\Adr}

\begin{abstract}
As a well-known enumerative problem, the number of solutions to the equation $m=m_1+\ldots+m_k$ with $m_1\leqslant\ldots\leqslant m_k$ in positive integers is $\Pi(m,k)=\sum_{i=0}^k\Pi(m-k,i)$ and $\Pi$ is called the additive partition function. In this paper, we give a recursive formula for the number of solutions to the equation $m=m_1\ldots m_k$ with $m_1\leqslant\ldots\leqslant m_k$ in positive integers. In particular, using elementary techniques, we give an explicit formula for the cases $k=1,2,3,4$.
\end{abstract}

\maketitle

\section{Introduction}
Let $\mathcal{F}(n;k,\ell)$ be the number of unordered factorizations of a positive
 integer $n$ to exactly $k$ parts, such that each parts $\geqslant\ell$.
  We denote the number of all unordered factorizations 
  of a positive integer $n$ by $\mathcal{F}(n)$, 
  that is the number of ways a positive integer $n$ 
  can be written as a product $n=n_1\times n_2\times\ldots\times  n_{k}$,
   where $n_1\geqslant n_2 \geqslant \ldots \geqslant n_{k} >1$. 
  The integers $n_1, n_2, \ldots, n_{k}$ are called the \textit{factors}
   of the factorization, and it's clearly that
    $\mathcal{F}(n)=\sum_{k=1}^n\mathcal{F}(n;k,2)$. We call
$\mathcal{F}(n)$ is \textit{the unordered Factorization function} of $n$. For example
 $\mathcal{F}(12)$, corresponding to
  $2\times6, 2\times2\times3, 3\times4$ and $12$.
   The sequence $\mathcal{F}(n)$ is listed in \cite{In1}.
    The Dirichlet generating function for $\mathcal{F}(n)$ is 
\[\prod_{k=2}^ \infty \frac{1}{1-k^{-s}}=\sum_{n=1}^{\infty}\frac{\mathcal{F}(n)}{n^s}.\]
 For positive integers $\ell, k\geqslant1$, we denote the number 
 of \textit{ordered factorization} of
 positive integer $n$ in exactly $k$ parts,
  such that each part $\geqslant\ell$ by $\mathcal{H}(n;k,\ell)$.
   We use $\mathcal{H}(n)$ to represent the number of all
    ordered factorization of the integer $n$ (in analogy with
     compositions for sum), then $\mathcal{H}(n)=\sum_{k=1}^n\mathcal{H}(n;k,2)$
     . An additive partition of a positive integer $n$ that denoted $p(n)$,
      is an integer $k$-tuple $n_1\geqslant n_2\geqslant \ldots\geqslant n_{k}>0$,
       for some $k$, such that $n=n_1+n_2+\ldots+n_{k}$ 
       (in analogy with factorization
function $\mathcal{F}(n)$ for product).
        The integers $n_1, n_2, \ldots, n_{k}$ are the \textit{parts} 
        of the partitions. For example $p(4)$ corresponding to, 
        $1+1+1+1, 1+1+2, 1+3, 2+2$ and $4$.
         It is important note that if
          $n=p_1^{\beta_1}p_2^{\beta_2}\ldots p_k^{\beta_k}$,
           where $p_1,p_2,\ldots,p_k$ are distinct
            prime numbers and $ \beta_i\in\mathbb{N}$ for
             $1\leqslant i\leqslant k$, then $\mathcal{F}(n)$ 
             and $\mathcal{H}(n)$ depend only to $\beta_1,\beta_2,\ldots,\beta_k$
             . For instance, if a positive integer $n$ is a
              prime power $n=p^k$, $k\geqslant1$, then
               $\mathcal{F}(n)=p(k)$, and $\mathcal{H}(n)=2^{k-1}$.
     Also, if a positive integer $n$ is square free as $n=p_1\times p_2 \times \ldots \times p_k$
               then $\mathcal{F}(n)=\sum_{i=1}^k {k \brace i}$, where ${n \brace k}$ is the \textit{Stirling number of the second kind}, and $\mathcal{H}(n)=\sum_{i=1}^k i! {k \brace i}$. 
               Let $\mathcal{F}(n;\{\beta_1,\ldots,\beta_r\},\ell)$, be the number of unordered factorizations of a positive integer $n$ as $n=n_1^{\beta_1}\times\ldots \times n_r^{\beta_r}$, such that $\beta_1+\ldots+\beta_r=k$ and $\ell\leqslant n_1<\ldots<n_r$, also, $\mathcal{H}(n;\{\beta_1,\ldots,\beta_r\},\ell)$ be the number of ordered factorizations of a positive integer $n$ as $n=n_1^{\beta_1}\times\ldots \times n_r^{\beta_r}$, such that $n_i\geqslant\ell$,
                 and $\{n_1,\ldots n_k\}=\{n'_1,\ldots,n'_r\}$ and $\beta_j=|\{i: n_i=n'_j\}|,$ for each $1\leqslant i,j\leqslant r$. For example, $\mathcal{F}(n;\{1,1,2\},\ell)$ is the number of unordered factorization positive integer $n$ as the form $xyz^2$, where $x, y$ and $z$ are different positive integers and $x>y>z\geqslant\ell$ and $\mathcal{H}(n;\{1,1,2\},\ell)=2!\mathcal{F}(n;\{1,1,2\},\ell)$. It is easy to see that
 \begin{eqnarray*}
\mathcal{F}(n;\{\beta_1,\ldots,\beta_r\},\ell)=\frac{( \beta_1+\ldots+\beta_r)!}{\beta_1!\ldots\beta_r!}\mathcal{H}(n;\{\beta_1,\ldots,\beta_r\},\ell).
 \end{eqnarray*}
  More on factorization partitions, including results on bounds and asymptotes of $\mathcal{F}(n)$ and algorithms for calculating the values, can be found in \cite{Can, Har, Knp} and \cite{Mat}.\\
The goal of this paper is to give some recursive formula for $\mathcal{F}(n)$ and $\mathcal{H}(n)$
 also we obtain $\mathcal{F}(n,k,\ell)$ and $\mathcal{H}(n,k,\ell)$ for cases $n=2,3,4$,  with elementary ways. Also, we give another proof of general formula for the number $\mathcal{H}_{\ell}(n,k)$ of ordered factorizations of a positive integer $n$ in exactly $k$ factors that each factor greater than $1$, was found in $1893$ by \textit{MacMahon} \cite{Mac}:
\[\mathcal{H}(n;k,2)=\sum_{i=0}^{k-1}(-1)^i{k\choose i}\prod_{j=1}^{n}{\beta_j+k-i-1\choose k-i-1}.\]
At the end, we closed our paper by posting several propositions about additive partition function $p(n)$.
\section{A Recursive Formula }
%
%
In this section we give some recursive formula $\mathcal{F}(n;k,\ell)$ and $\mathcal{H}(n;k,\ell)$. Let $n=n_1^{\beta_1}\times\ldots \times n_r^{\beta_r}$ be a positive integer, where $\beta_i\in\mathbb{N}$. By using of above notations, we can write
\begin{eqnarray}
\mathcal{F}(n;k,\ell)=\sum_{\begin{subarray}{c}\beta_1+\ldots+\beta_r=k;\\
\beta_1<\ldots<\beta_r,\end{subarray}}\mathcal{F}(n;\{\beta_1,\ldots,\beta_r\},\ell);
\end{eqnarray}
and
\begin{eqnarray}
\mathcal{H}(n;k,\ell)=\sum_{\beta_1+\ldots+\beta_r=k}\mathcal{H}(n;\{\beta_1,\ldots,\beta_r\},\ell).
\end{eqnarray}

\begin{proposition}\label{ltolplus}
Let $n>1$ and $k,\ell$ be positive integers. Suppose that $\ell^s$ divides $n$ but $\ell^{s+1}$ does not divide $n$. Then
\[\mathcal{F}(n;k,\ell)=\sum_{i=\max\{k-s,1\}}^{\min\{k,s\}}\mathcal{F}(n;i,\ell+1).\]
\end{proposition}
\begin{proof}
Let
\[E=\{(n_1,n_2,\ldots,n_k): n=n_1\times n_2\times\ldots \times n_k, \ell\leqslant n_1\leqslant \ldots\leqslant n_k\}\]
and
\[E_i=\{(n_1,n_2,\ldots,n_k)\in E: n_1=n_2=\ldots=n_i=\ell, n_{i+1}\neq \ell\},\quad i=0,1,\ldots,\min\{k-1,s\}.\]
Then $E=\cup_{i=0}^{\min\{k-1,s\}}E_i$ and the union is disjoint. Thus
\[\mathcal{F}(n;k,\ell)=\sum_{i=0}^{\min\{k-1,s\}}|E_i|=\sum_{i=0}^{\min\{k-1,s\}}\mathcal{F}(n;k-i,\ell+1)=\sum_{i=\max\{k-s,1\}}^{\min\{k,s\}}\mathcal{F}(n;i,\ell+1).\]
\end{proof}
\begin{corollary}\label{1to2}
Let $n>1$ and $k,\ell$ be positive integers. Then
\[\mathcal{F}(n;k,1)=\sum_{i=1}^k\mathcal{F}(n;i,2).\]
\end{corollary}
\begin{proof}
Since for each positive integer $s$ we have $1^s\mid n$, the implication is a result of Proposition~\ref{ltolplus}.
\end{proof}

We end this section with the following result whose proof is straightforward.
\begin{lemma}
Let $n, k$ and $\ell$ be positive integers. Then
\begin{eqnarray*}
\mathcal{F}(n;k,\ell)=\sum_{\ell\leqslant d\mid n}\mathcal{F}_d(\frac{n}{d},k-1).
\end{eqnarray*}
\end{lemma}
\begin{proof}
For positive integer $n$ we define $\mathcal{F}_\ell(1,k)=1$ and $\mathcal{F}_\ell(n,0)=0$. Let $n=n_1\times n_2\times\ldots\times n_k$ be a factorization partition of $n$ that $n_1\geqslant n_2\geqslant \ldots\geqslant n_k\geqslant \ell >0$. Then
 $\frac{n}{n_1}=n_2\times n_3\times\ldots\times n_k$ is the factorization partitions of $\frac{n}{n_1}$ in $k-1$ factor that each factor greater than $\ell$, and the number of all such factorization partition is $\mathcal{F}_\ell(\frac{n}{n_1},k)$. Since $n_1$ was unspecified, therefore, we can write such factorization partition for any divisor $d\leqslant \ell$ of $n$. Summation over all such divisor $d$ of $n$ gives the proof.
\end{proof}
\begin{lemma}
Let $n, k$ and $\ell$ be positive integers. Then
\begin{eqnarray*}
\mathcal{H}(n;k,\ell)=\sum_{\ell\leqslant d\mid n}\mathcal{H}_d(\frac{n}{d},k-1).
\end{eqnarray*}
\end{lemma}


\section{An Explicit Formula For the Cases $k=1,2,3,4$ and $\ell=1,2$}
We are now intended to give an explicit formula for the $\mathcal{F}(n;k,\ell)$ and $\mathcal{H}(n;k,\ell)$ when $k=1,2,3,4$ and $\ell=1,2$. In the following proposition, we show the number of natural divisors of $n$ by $\tau(n)$. Moreover,
\[\varepsilon_i(n)=\left\{\begin{array}{ll}
1 & \mbox{if~}\sqrt[i]{n}\in\mathbb{N}\\
0 & \mbox{otherwise~}
\end{array}\right. \]
\begin{proposition}
Let $n>1$ be a positive integer. Then
\begin{itemize}
\item[i.] $\mathcal{F}(n;1,1)=\mathcal{F}(n,1,2)=1$,
\item[ii.] $\mathcal{F}(n;2,1) =\lceil\frac{\tau(n)}{2}\rceil\mbox{~and~}\mathcal{F}(n;2,2)=\lceil\frac{\tau(n)}{2}\rceil-1$.
\end{itemize}
\end{proposition}
\begin{proof}
The equalities in item (i) are obvious. To prove item (ii), we note that $\mathcal{H}(n;2,1)$ is number of ways to write $n$ as $xy$, where $x$ is a natural divisor of $n$. Thus $\mathcal{H}(n;2,1)=\tau(n)$. Now if $n$ is not a perfect square then $\tau(n)$ is even and so $\mathcal{F}(n;2,1)=\frac{\tau(n)}{2}=\lceil\frac{\tau(n)}{2}\rceil$ and if $n$ is a perfect square then $\mathcal{F}(n;2,1)=\frac{\tau(n)-1}{2}+1=\lceil\frac{\tau(n)}{2}\rceil$. Using Corollary~\ref{1to2}, we now have $\mathcal{F}(n;2,2)=\mathcal{F}(n;2,1)-1$.
\end{proof}
\begin{theorem}
Let $n>1$ be a positive integer and $p_1^{\beta_1}\ldots p_n^{\beta_n}$ be its prime decomposition. Then
\begin{itemize}
\item[i.] $\mathcal{F}(n;3,1)=\frac{1}{6}\prod_{j=1}^n{\beta_j+2\choose 2}+\frac{1}{2}\prod_{j=1}^n\lfloor\frac{\beta_j+2}{2}\rfloor+\frac{\varepsilon_3(n)}{3}$,
\item[ii.]$\mathcal{F}(n;3,2) =\mathcal{F}(n;3,1)-\lceil\frac{\tau(n)}{2}\rceil$.
\end{itemize}
\end{theorem}
\begin{proof}
We have
\begin{eqnarray*}
\mathcal{H}(n;3,1)&=&\mathcal{H}(n;\{1,1,1\},1)+\mathcal{H}(n;\{1,2\},1)+\mathcal{H}(n;\{3\},1)\\
&=&6\mathcal{F}(n;\{1,1,1\},1)+3\mathcal{F}(n;\{1,2\},1)+\mathcal{F}(n;\{3\},1).
\end{eqnarray*}
We know that $\mathcal{F}(n;\{1,2\},1)$ is the number of ways to write $n$ as $xy^2$, where $x\neq y$. This is equal to the number of $y$'s such that $y^2\mid n$ minus the number of ways such that $\frac{n}{y^2}=y$, in which the later is equal to $\varepsilon_3(n)$. The number of $y$'s such that $y^2\mid n$ is $\prod_{j=1}^n\lfloor\frac{\beta_j+2}2\rfloor$. Moreover, $\mathcal{F}(n;\{3\},1)=\varepsilon_3(n)$. Thus Lemma~\ref{nu1} implies that
\begin{align*}
\mathcal{F}(n;\{1,1,1\},1)&=\frac{1}{6}\big(\mathcal{H}(n;\{3\},1)-3\big(\prod_{j=1}^n\lfloor\frac{\beta_j+2}2\rfloor-\varepsilon_3(n)\big)-\varepsilon_3(n)\big)\\
&=\frac{1}{6}\mathcal{H}(n;\{3\},1)-\frac{1}{2}\prod_{j=1}^n\lfloor\frac{\beta_j+2}2\rfloor+\frac{1}{3}\varepsilon_3(n)\\
&=\frac{1}{6}\prod_{j=1}^n{\beta_j+2\choose 2}-\frac{1}{2}\prod_{j=1}^n\lfloor\frac{\beta_j+2}2\rfloor+\frac{\varepsilon_3(n)}{3}.
\end{align*}
Therefore, we must have
\begin{align*}
\mathcal{F}(n;3,1)&=\mathcal{F}(n;\{1,1,1\},1)+\mathcal{F}(n;\{1,2\},1)+\mathcal{F}(n;\{3\},1)\\
&=\frac{1}{6}\prod_{j=1}^n{\beta_j+2\choose 2}-\frac{1}{2}\prod_{j=1}^n\lfloor\frac{\beta_j+2}2\rfloor+\frac{\varepsilon_3(n)}{3}+\prod_{j=1}^n\lfloor\frac{\beta_j+2}2\rfloor-\varepsilon_3(n)+\varepsilon_3(n).
\end{align*}
This proves item (i).

Using Corollary~\ref{1to2}, we have
\[\mathcal{F}(n;3,2)=\mathcal{F}(n;3,1)-\mathcal{F}(n;2,2)-1=\mathcal{F}(n;3,1)-\lceil\frac{\tau(n)}{2}\rceil+1-1\]
which proves item (ii).
\end{proof}
the following lemma is also easy to prove.
\begin{lemma}\label{sumtau}
Let $k,\ell>1$ be a positive integer and $p_1^{\beta_1}\ldots p_n^{\beta_n}$ be  the prime decomposition of $k$. Then
\[\sum_{d\mid k}\tau(d)=\prod_{j=1}^n{\beta_j+2\choose 2}\]
and
\[\sum_{d\mid k}\varepsilon_\ell(d)=\prod_{j=1}^n\lfloor\frac{\beta_j+\ell}{\ell}\rfloor.\]
\end{lemma}
\begin{theorem}
Let $n>1$ be a positive integer and $p_1^{\beta_1}\ldots p_n^{\beta_n}$ be its prime decomposition. then
\begin{align*}
\mathcal{F}(n;4,1)&=\frac{1}{24}\prod_{i=1}^n{\beta_i+3\choose 3}+\frac{1}{3}\prod_{i=1}^n\lfloor\frac{\beta_i+3}3\rfloor
+\frac{1}{4}(\prod_{i=1}^n \lfloor\frac{\beta_i+2}{2}\rfloor(\beta_i-\lfloor\frac{\beta_i-2}{2}\rfloor))\\
&\,\,+\frac{\varepsilon_2(n)}4\prod_{i=1}^n\lfloor\frac{\beta_i+2}2\rfloor- \frac{\varepsilon_2(n)}4\lceil\frac{\tau(\sqrt{n})}{2}\rceil+\frac{3\varepsilon_4(n)}8.
\end{align*}
Moreover,
\[\mathcal{F}(n;4,2)=\mathcal{F}(n;4,1)-\mathcal{F}(n;3,1).\]
\end{theorem}
\begin{proof}
We have
\begin{align*}
\mathcal{H}(n;4,1)&=\mathcal{H}(n;\{1,1,1,1\},1)+\mathcal{H}(n;\{1,1,2\},1)\\
&+\mathcal{H}(n;\{1,3\},1)+\mathcal{H}(n;\{2,2\},1)+\mathcal{H}(n;\{4\})\\
&=24\mathcal{F}(n;\{1,1,1,1\},1)+12\mathcal{F}(n;\{1,1,2\},1)\\
&+4\mathcal{F}(n;\{1,3\},1)+6\mathcal{F}(n;\{2,2\},1)+\mathcal{F}(n;\{4\},1).
\end{align*}
On the other hand, $\mathcal{F}(n;\{1,1,2\},1)$ is the number of ways to write $n$ as $xyz^2$, where $x, y$ and $z$ are different positive integers. This is equal to the number of $z$'s such that $z^2\mid n$ minus the number of ways to write $n$ as $xz^3, x^2z^2$ or $z^4$, where $x\neq z$. The number of $z$'s such that $z^2\mid n$ is $\sum_{ z^2\mid n}\lceil{\frac{\tau(\frac{n}{z^2})}{2}}\rceil$. Thus
\begin{align*}
\mathcal{F}(n;\{1,1,2\},1)=\sum_{ z^2\mid n}\lceil{\frac{\tau(\frac{n}{z^2})}{2}}\rceil-\mathcal{F}(n;\{1,3\},1)-\mathcal{F}(n;\{2,2\},1)-\varepsilon_4(n).
\end{align*}
If $n$ is not a perfect square then $\tau(n)$ is even. Let $z=p_1^{\beta_1}\ldots p_n^{\beta_n}$. So, there must exist an integer $i$ such that $\beta_i$ is odd. Hence
\begin{align*}
\sum_{ z^2\mid n}\lceil{\frac{\tau(\frac{n}{z^2})}{2}}\rceil&=\sum_{\ z^2\mid n}\frac{1}{2}\tau(p_1^{\beta_1-2\beta_1}\ldots p_n^{\beta_n-2\beta_n})
 \\&=\frac{1}{2}\prod_{i=1}^n\sum_{0\leq {\beta_i}\leq\frac{\beta_i}{2}}(\beta_i-2\beta_i+1)
 \\&=\frac{1}{2}(\prod_{i=1}^n \lfloor\frac{\beta_i+2}{2}\rfloor ( \beta_i-\lfloor\frac{\beta_i-2}{2}\rfloor)).
 \end{align*}
Moreover, $\mathcal{F}(n;\{1,3\},1)$ is the number of ways to write $n$ as $xz^3$, where $x\neq z$. This is equal to the number of $z$'s such that $z^3\mid n$ minus the number of ways to write $n$ as $z^4$. The number of $z$'s such that $z^3\mid n$ is $\prod_{j=1}^n\lfloor\frac{\beta_j+3}3\rfloor$. Whence
\[\mathcal{F}(n;\{1,3\},1)=\prod_{j=1}^n\lfloor\frac{\beta_j+3}3\rfloor-\varepsilon_4(n).\]
Since $n$ is not a perfect square, $\mathcal{F}(n;\{2,2\},1)=0$. Thus implies
\begin{align*}
\mathcal{F}(n;\{1,1,1,1\{,1)&=\frac{1}{24}\big(\mathcal{H}(n;\{4\},1)
-12\mathcal{F}(n;\{1,1,2\},1)-4\mathcal{F}(n;\{1,3\{,1)-6\mathcal{F}(n;\{2,2\},1)-\varepsilon_4(n)\big)\\
&=\frac{1}{24} \prod_{i=1}^n{\beta_i+3\choose 3}-
\frac{1}{4}(\prod_{i=1}^n \lfloor\frac{\beta_i+2}{2}\rfloor ( \beta_i-\lfloor\frac{\beta_i-2}{2}\rfloor))
\\&\,\,+ \frac{1}{3}\mathcal{F}(n;\{1,3\},1)+ \frac{1}{4}\mathcal{F}(n;\{2,2\},1)+ \frac{11}{24}\varepsilon_4(n)\big)
\\&=\frac{1}{24} \prod_{i=1}^n{\beta_i+3\choose 3}-
\frac{1}{4}(\prod_{i=1}^n \lfloor\frac{\beta_i+2}{2}\rfloor ( \beta_i-\lfloor\frac{\beta_i-2}{2}\rfloor))
+ \frac{1}{3}\mathcal{F}(n;\{1,3\},1).
\end{align*}
Therefore, we have
\begin{align*}
\mathcal{F}(n;4,1)&=\mathcal{F}(n;\{1,1,1,1\},1)+\mathcal{F}(n;\{1,1,2\},1)+\mathcal{F}(n;\{1,3\},1)+\mathcal{F}(n;\{2,2\},1)+\mathcal{F}(n;\{4\},1)\\
&=\frac{1}{24}\prod_{i=1}^n{\beta_i+3\choose 3}+\frac{1}{3}\prod_{i=1}^n\lfloor\frac{\beta_i+3}3\rfloor+\frac{1}{4}(\prod_{i=1}^n \lfloor\frac{\beta_i+2}{2}\rfloor ( \beta_i-\lfloor\frac{\beta_i-2}{2}\rfloor)).
\end{align*}
Now let $n$ be a perfect square. Then $\tau(n)$ is odd and we have
\begin{align*}
\sum_{ z^2\mid n}\lceil{\frac{\tau(\frac{n}{d^2})}{2}}\rceil&=\sum_{ z^2\mid n}(\frac{\tau(\frac{m}{d^2})+1}{2})
\\&=\sum_{\ z^2\mid n}\frac{1}{2}\tau(p_1^{\beta_1-2\beta_1}\ldots p_n^{\beta_n-2\beta_n})+ \frac{1}{2}\sum_{ z^2\mid n}1
 \\&=\frac{1}{2}\sum_{0\leq {\beta_i}\leq\frac{\beta_i}{2}}\prod_{i=1}^n(\beta_i-2\beta_i+1)+\frac{1}{2}\sum_{0\leq {\beta_i}\leq\frac{\beta_i}{2}} 1
 \\&=\frac{1}{2}\prod_{i=1}^n\sum_{0\leq {\beta_i}\leq\frac{\beta_i}{2}}(\beta_i-2\beta_i+1)
+ \frac{1}{2}\prod_{i=1}^n (\lfloor\frac{\beta_i+2}{2}\rfloor)
\\&=\frac{1}{2}(\prod_{i=1}^n \lfloor\frac{\beta_i+2}{2}\rfloor ( \beta_i-\lfloor\frac{\beta_i-2}{2}\rfloor))
+\frac{1}{2}\prod_{i=1}^n (\lfloor\frac{\beta_i+2}{2}\rfloor)
\end{align*}
Thus
\begin{align*}
 \mathcal{F}(n;\{1,1,1,1\},1)&=\frac{1}{24} \prod_{i=1}^n{\beta_i+3\choose 3}
 -\frac{1}{4}(\prod_{i=1}^n \lfloor\frac{\beta_i+2}{2}\rfloor ( \beta_i-\lfloor\frac{\beta_i-2}{2}\rfloor))
\\&+\frac{1}{2}\prod_{i=1}^n (\lfloor\frac{\beta_i+2}{2}\rfloor)+ \frac{1}{3}\mathcal{F}(n;\{1,3\},1)
- \frac{1}{4}\mathcal{F}(n;\{2,2\},1)- \frac{1}{24}\varepsilon_4(n).
\end{align*}
Furthermore, $\mathcal{F}(n;\{2,2\},1)$ is the number of ways to write $n$ as $x^2y^2=(xy)^2$, where $x\neq y$. If $n$ is not a perfect square then this number is 0 and if $n$ is a perfect square then $\sqrt{n}\in\mathbb{N}$ and thus
\begin{eqnarray*}
\mathcal{F}(n;\{2,2\},1)&=&\varepsilon_2(n)\mathcal{F}(\sqrt{n};\{1,1\},2)=\varepsilon_2(n)\big(\mathcal{F}(\sqrt{n},\{2\},2)-\varepsilon_2(\sqrt{n})\big)\\
&=&\varepsilon_2(n)\big(\lceil\frac{\tau(\sqrt{n})}{2}\rceil-\varepsilon_4(n)\big)\\
&=&\varepsilon_2(n)\lceil\frac{\tau(\sqrt{n})}{2}\rceil-\varepsilon_4(n).
\end{eqnarray*}
Therefore
\begin{align*}
\mathcal{F}(n;4,1)&=\frac{1}{24}\prod_{i=1}^n{\beta_i+3\choose 3}+\frac{1}{3}\prod_{i=1}^n\lfloor\frac{\beta_i+3}3\rfloor
+\frac{1}{4}(\prod_{i=1}^n \lfloor\frac{\beta_i+2}{2}\rfloor(\beta_i-\lfloor\frac{\beta_i-2}{2}\rfloor))\\
&\,\,+\frac{\varepsilon_2(n)}4\prod_{i=1}^n\lfloor\frac{\beta_i+2}2\rfloor- \frac{\varepsilon_2(n)}4\lceil\frac{\tau(\sqrt{n})}{2}\rceil+\frac{3\varepsilon_4(n)}8.
\end{align*}
This proves the first assertion. The second part can be driven from the first assertion.
\end{proof}

 Note that if a positive integer $m$ is a prime power, say $m=p^n$, then $\mathcal{F}_2(m,k)=\rho(n,k)$, where $\rho$ is the additive partition function.
 \begin{corollary}
 Let $n$ be a positive integer. Then the number of all solutions to the equation
 \[x_{1}+x_{2}+x_3=n\]
 in $\mathbb{N}$, under the condition $x_1\leq x_2 \leq x_3$, is
\[\frac{1}{6}{n+2\choose 2}+\frac{1}{2}\lfloor\frac{n+2}{2}\rfloor+\frac{\varepsilon_3(p^n)}{3}.\]
 \end{corollary}
 \begin{corollary}
 Let $n$ be positive a integer. Then the number of all solutions of the equation
 \[x_{1}+x_{2}+x_{3}+x_4=n\]
 in $\mathbb{N}$, under the condition $x_1\leq x_2 \leq x_3\leq x_4$, is
 \begin{eqnarray*}&&\frac{1}{24}{n+3\choose 3}+\frac{1}{3}\lfloor\frac{n+3}3\rfloor
+\frac{1}{4}(\lfloor\frac{n+2}{2}\rfloor ( n-\lfloor\frac{n-2}{2}\rfloor))
\\&+&\frac{\varepsilon_2(p^n)}4\lfloor\frac{n+2}2\rfloor-
 \frac{\varepsilon_2(p^n)}4\lceil\frac{\tau(\sqrt{p^n})}{2}\rceil
 +\frac{3\varepsilon_4(p^n)}8.\end{eqnarray*}
 \end{corollary}

\end{document}